\documentclass[10pt, reqno]{amsart}
\usepackage[utf8]{inputenc}
\usepackage[english]{babel}
\usepackage[T1]{fontenc}
\usepackage{graphicx,caption,subcaption}
\usepackage{amsmath,amsthm, amssymb,mathrsfs}
\usepackage{mathabx,epsfig}
\usepackage{todonotes}
\usepackage{comment}
\usepackage[colorlinks=true, allcolors=blue]{hyperref}
\usepackage{enumitem}

\usepackage{tikz-cd}
\usepackage{tikz}
\usetikzlibrary{matrix}

\def\from{\colon}

\def\Julia{\mathcal{J}}

\def\Rat{\mathrm{Rat}}
\def\QB{\mathcal{QB}}

\def\Cbb{\mathbb{C}}

\def\Rbb{\mathbb{R}}

\def\Bcal{\mathcal{B}}

\def\Hcal{\mathcal{H}}

\def\Pcal{\mathcal{P}}

\newcommand{\act}{\reflectbox{$\righttoleftarrow$}}

\DeclareMathOperator{\PSL}{PSL}

\newtheorem{thm}{Theorem}[section]

\newtheorem{lem}[thm]{Lemma}
\newtheorem{prop}[thm]{Proposition}

\newtheorem{theorem}{Theorem}

\theoremstyle{definition}
\newtheorem{defn}[thm]{Definition}

\newtheorem{rem}[thm]{Remark}

\title[Analytic theory on the space of Blaschke products]{Analytic theory on the space of Blaschke products: Simultaneous uniformization and Weil-Petersson metric}
\author{Yan Mary He}
\address{Department of Mathematics, University of Oklahoma, Norman, OK 73019, USA}
\email{he@ou.edu}
\author{Homin Lee}
\address{Center for Mathematical Challenges, Korea Institute for Advanced Study,  Seoul, Republic of Korea}
\email{hominlee@kias.re.kr}
\author{Insung Park}
\address{Institute for Mathematical Sciences, Stony Brook University, Stony Brook, NY 11794, USA} 
\email{insung.park@stonybrook.edu}
\date{\today\\ MSC Class: 37F10; 37F31; 32G15}

\begin{document}

\begin{abstract}
In this paper, we study complex analytic aspects of the moduli space $\Bcal_d^{fm}$ of degree $d\ge2$ fixed-point-marked Blaschke products. We define a complex structure on $\Bcal_d^{fm}$ and prove the simultaneous uniformization theorem for fixed-point-marked quasi-Blaschke products. As an application, we show that the Weil-Petersson semi-norm (which is also referred to as the pressure semi-norm) is non-degenerate along all directions transverse to a holomorphic codimension-1 foliation  
of $\Bcal^{fm}_d$.
\end{abstract}

\maketitle
\section{Introduction}
On the Riemann sphere $\mathbb{P}^1$, the dynamics of rational maps and those of Kleinian groups exhibit many analogous phenomena, which are known as Sullivan's dictionary. Within this framework, McMullen has shown that Blaschke products and Fuchsian groups, as well as their respective deformation spaces, share many similar dynamical properties \cite{McMullen08,McM_simul,McM_CompExpCircle,McM_DynDisk}. Very recently, Luo has extended this analogy to quasi-Blaschke products and quasi-Fuchsian groups, and has proved various results for quasi-Blaschke products which have counterparts in the 3-dimensional hyperbolic geometry \cite{Luo_GeoFiniteDegenI, Luo_GeoFiniteDegenII}. 
In an earlier paper \cite{HLP23}, we studied the degenerate locus of the {\it pressure semi-norm} on the space $\QB_d^{fm}$ of degree $d \ge 2$ fixed-point-marked quasi-Blaschke products, which parallels the work of Bridgeman \cite{Bridgeman_WPMetricQF} in the context of quasi-Fuchsian groups. 
In this paper, we further develop the analogies between (quasi-)Blaschke products and (quasi-)Fuchsian groups.

In \cite{McMullen08}, McMullen defined the {\it Weil--Petersson semi-norm} $||\cdot||_{WP}$ on the space $\Bcal_d^{fm}$ of degree $d \ge 2$ fixed-point-marked Blaschke products. We refer to it as the Weil–Petersson semi-norm because the analogous construction for Fuchsian groups gives a constant multiple of the Weil–Petersson norm; see  \cite[Theorem 1.12]{McMullen08}. McMullen also showed that there are several equivalent ways to construct the semi-norm $||\cdot||_{WP}$ on $\Bcal_d^{fm}$; see \cite[Theorems 1.6 and 1.7]{McMullen08}.
However, it remains unknown whether this semi-norm is non-degenerate, namely, if there exist $[f] \in \Bcal_d^{fm}$ and a non-zero tangent vector $\vec{v} \in T_{[f]}\Bcal_d^{fm}$ with $||\vec{v}||_{WP}=0$.
The main goal of this paper is to address this non-degeneracy question.

As a byproduct of proving the non-degeneracy of the semi-norm $||\cdot||_{WP}$, we obtain a simultaneous uniformization theorem for quasi-Blaschke products, which is analogous to Bers' 
simultaneous uniformization theorem for quasi-Fuchsian groups and is interesting in its own right.

\subsection{Pressure semi-norms and degenerating vectors}
Let $d\ge2$ be an integer. To avoid orbifold singularities, we consider {\it fixed-point-marked rational maps}, which are rational maps together with ordered $(d+1)$-tuples of their fixed points $(f;x_1,x_2,\dots,x_{d+1})$. We denote by ${\rm rat}^{fm}_d$ the moduli space of fixed-point-marked rational maps.
A hyperbolic rational map $f$ is a {\it quasi-Blaschke product} if its Julia set $\Julia
(f)$ is a quasicircle and $f$ fixes each of the two Fatou components. A quasi-Blaschke product is a {\it Blaschke product} if $\Julia(f)$ is smooth (i.e., $\Julia(f)$ is the unit circle).
We consider the moduli space of fixed-point-marked quasi-Blaschke products (resp.\@ Blaschke products) $\QB^{fm}_d$ (resp.\@ $\Bcal^{fm}_d$); see Section \ref{sec_defQBd} for a detailed discussion.

Fix $[f] \in \mathcal{B}_d^{fm}$. Then $f$ has two attracting or super-attracting fixed points whose multipliers, which are the derivatives $f'$ at the fixed points, are complex conjugate to each other. One of them, denoted by $\lambda_{h.att}([f])$, varies holomorphically with $[f]$ with respect to the holomorphic structure on $\mathcal{B}_d^{fm}$ which we will discuss in Section \ref{sec_intro_1.2}. We call $\lambda_{h.att}([f])$ the \emph{holomorphic attracting multiplier} of $[f]$; see Definition~\ref{def_anti_hol_fixed}.

Let $\Delta:=\{z\in\Cbb : |z|<1\}$.
The holomorphic attracting multiplier defines a holomorphic map
\[
\lambda_{\mathrm{h.att}} \colon \mathcal B_d^{fm} \to \Delta,
    \qquad
    [f] \mapsto \lambda_{\mathrm{h.att}}([f]);
\]
see Section \ref{sec_defQBd} for the detailed definition.
For each $\lambda \in \Delta$, the level set
$\mathcal A_d^{fm}(\lambda)$ is a holomorphic
codimension-one submanifold of $\mathcal B_d^{fm}$.
The following theorem shows that the Weil--Petersson semi-norm is non-degenerate in every direction transverse to these level sets and is non-degenerate almost everywhere in directions tangent to them. When $d=2$, in particular, we obtain the non-degeneracy of $||\cdot||_{WP}$ everywhere.

\begin{theorem}\label{main theorem 2}
For $d\ge 2$ and $\lambda\in\Delta$, we define
\begin{equation}\label{eq_intro_SA}
    \mathcal{A}^{fm}_d(\lambda) := \left\{[f] \in \Bcal_d^{fm} : \lambda_{h.att}([f])=\lambda \right\}.  
\end{equation}
Consider $T\mathcal{A}_d^{fm}:=\bigcup_{\lambda\in \Delta} T\mathcal{A}^{fm}_d(\lambda)$ as a sub-bundle of $T\Bcal_d^{fm}$.
Then the following hold.
\begin{enumerate}
\item For all non-zero vector $\vec{v} \in T\Bcal_d^{fm}\setminus T\mathcal{A}_d^{fm}$, we have $||\vec{v}||_{WP} \neq 0$.

\item For Lebesgue almost every $\vec{v}\in T\mathcal{A}_d^{fm}$, we have $||\vec{v}||_{WP}\neq0$.
\end{enumerate}
\end{theorem}

When $d=2$, the set $\mathcal{A}^{fm}_2(\lambda)$ is a singleton for each $\lambda\in \Delta$. Hence, we obtain the following theorem as an immediate corollary of Theorem \ref{main theorem 2}.

\begin{theorem}\label{main theorem 3}
The Weil--Petersson semi-norm $||\cdot ||_{WP}$ is non-degenerate on $\Bcal_2^{fm}$.
\end{theorem}

In Proposition \ref{prop:non-deg for deg3}, we show the non-degeneracy of $||\cdot||_{WP}$ on $\mathcal{B}_3^{fm}$ outside the locus where each attracting fixed point has exactly two points in its preimage, including itself.

The Weil--Petersson semi-norm $||\cdot||_{WP}$ is a constant multiple of the so-called {\it pressure semi-norm} $||\cdot||_{\mathcal{P}}$, more precisely $||\cdot||_{WP} =\frac{1}{2}||\cdot||_{\Pcal}$; see Section \ref{sec_Curt's_metric} for details. In general, the pressure semi-norms $||\cdot||_{\Pcal}$ can be defined on any hyperbolic component of ${\rm rat}^{fm}_d$, such as the space of quasi-Blaschke product $\QB^{fm}_d$; see \cite{HeNie_MetricHypComp,HLP23}. In \cite[Theorem 3.8]{HLP23}, we characterized the degeneracy loci of the pressure semi-norms $||\cdot||_{\Pcal}$ on the space of quasi-Blaschke products $\QB^{fm}_d$ under the assumption that $||\cdot||_{\mathcal{P}}$ on $\Bcal_d^{fm}$ is non-degenerate. This characterization is motivated, via Sullivan’s dictionary, by \cite[Main Theorem]{Bridgeman_WPMetricQF}. 

We say that Blaschke products satisfy the {\it infinitesimal marked multiplier rigidity} if the following holds: For a smooth family of Blaschke product $(f_t)_{t\in (-1,1)}$, if $\left.\frac{d}{dt}\right|_{t=0}\lambda_C(f_t)=0$ for every repelling multiplier $\lambda_C$, then $(f_t)_{t\in (-1,1)}$ is tangent at $t=0$ to the locus of $\PSL(2,\Cbb)$-conjugations of $f_0$. See Section \ref{sec_defQBd} for functions of repelling multipliers $\lambda_C(\cdot)$.
The non-degeneracy of $||\cdot||_{\mathcal{P}}$ is equivalent to the infinitesimal marked multiplier rigidity (Lemma \ref{lem_degen}).

In \cite[Section~5.7]{BH24}, Bianchi and the first-named author proved that if $||\cdot||$ is an analytic semi-norm on a manifold $M$ and $\int_0^1 ||\gamma'(t)||\, dt > 0$ for every $C^1$-curve $\gamma \colon [0,1] \to M$, then $||\cdot||$ defines a path metric on $M$. We verify that our Weil--Petersson semi-norm $||\cdot||_{WP}$ satisfies the hypothesis of \cite[Section~5.7]{BH24} and obtain the following theorem.

\begin{theorem}\label{theorem:Path metric}
    The Weil--Petersson semi-norm $||\cdot||_{WP}$ defines a path metric $d_{WP}$ on $\mathcal{B}_d^{fm}$. In particular, for any $[f], [g] \in \mathcal{B}_d^{fm}$, we have
    \[
        \inf_{\gamma} \int_{\gamma} ||\gamma'(t)||_{WP}\, dt > 0,
    \]
    where the infimum is taken over all $C^1$-curves $\gamma$ connecting $[f]$ and $[g]$.
\end{theorem}

When $d=2$, Ivrii proved that the Weil--Petersson metric is incomplete
and that its completion contains all geometrically finite parameters on
the boundary of the main cardioid. He also conjectured that these exhaust the
metric boundary \cite{Ivrii}.

\subsection{Simultaneous uniformization of quasi-Blaschke products} \label{sec_intro_1.2}
Our proof of Theorem \ref{main theorem 2} relies on a holomorphic structure $\Bcal_d^{fm}$. With respect to the natural holomorphic structure on $\QB_d^{fm}$ induced from the coefficients of rational maps, the Blaschke locus $\Bcal_d^{fm}$ is a non-holomorphic real analytic submanifold of $\QB_d^{fm}$. We adopt Bers’ idea to endow $\Bcal_d^{fm}$ with a holomorphic structure. In this context, it is natural to discuss the simultaneous uniformization of quasi-Blaschke products.

If $S$ is a closed orientable surface of genus $g \ge 2$, the {\it Teichm\"uller space} $T(S)$ of $S$ is the space of isotopy classes of hyperbolic structures on $S$, and the {\it quasi-Fuchsian space} $\mathcal{QF}(S)$ of $S$ is the space of quasiconformal deformations of the Fuchsian group $\rho(\pi_1(S))$, where $\rho \in T(S)$. In the 1960s, Bers proved a remarkable theorem that the space $\mathcal{QF}(S)$ is biholomorphic to the product $T(S) \times T(\overline{S})$ where $\overline{S}$ is the same surface as $S$ with opposite orientation \cite{Bers_SimultaneousUniform}. Bers' theorem states that a quasi-Fuchsian group can be simultaneously uniformized by two Fuchsian groups. Analogously, we establish a simultaneous uniformization for quasi-Blaschke products as follows.

\begin{theorem}\label{main theorem 1}
There exists a biholomorphism
    \[
        \mathcal{U} \from \Bcal_d^{fm} \times \overline{\Bcal_d^{fm}} \to \QB_d^{fm}
    \]
    such that for any fixed-point-marked Blaschke products $f$ and $g$, $\mathcal{U}([f],[g])$ restricted to its two Fatou components are biholomorphically conjugate to $f|_\Delta$ and $g|_{1/\Delta}$, respectively, where $\Delta:=\{z\in \mathbb{P}^1:|z|<1\}$ and $1/\Delta:=\{z\in \mathbb{P}^1:|z|>1\}$. 
\end{theorem} 

Note that the statement of Theorem \ref{main theorem 1} relies on the holomorphic structure of $\Bcal_d^{fm}$. One way to define a holomorphic structure on $\Bcal_d^{fm}$ is to use a diffeomorphism between $\Bcal_d^{fm}$ and the main hyperbolic component $\Hcal_d^{fm}$ in ${\rm poly}_d^{fm}$, which is the hyperbolic component containing $(z^d;0,\infty,1, \zeta,\zeta^2,\dots,\zeta^{d-2})$, where $\zeta=e^{2 \pi i \frac{1}{d-1}}$. It is known that $\Hcal_d^{fm}$ is a complex manifold (see \cite{Milnor_HypComp}), and thus we can pull its holomorphic structure back to $\Bcal_d^{fm}$. It might be possible to construct a diffeomorphism between $\Bcal_d^{fm}$ and $\Hcal_d^{fm}$ directly by using quasiconformal surgeries. However, in this paper, we first prove a simultaneous uniformization of quasi-Blaschke products into a pair of polynomials in $\Hcal_d^{fm}$; see Theorem \ref{thm_simul}. Then, as corollaries, we obtain the desired diffeomorphism $\Bcal_d^{fm}\to \Hcal_d^{fm}$ as well as Theorem \ref{main theorem 1}. The hyperbolic component $\Hcal_d^{fm}$ is an analogue of the Bers slice in the classical theory of Teichm\"uller space. Under Bers' simultaneous uniformization
$\mathcal T(S)\times\mathcal T(\overline S)\to
\mathcal{QF}(S)$,
a {\it Bers slice} is obtained by fixing one of the two factors.

In the 1980s, McMullen introduced a method for gluing two Blaschke products of the same degree into a quasi-Blaschke product via conformal welding \cite{McM_simul}; also see \cite[Proposition 5.5]{McM_AutRat}. This construction can be interpreted as the {\it mating} of two Blaschke products—--a term commonly used in complex dynamics to describe the combination of two dynamical systems. However, 
compared to McMullen's work, our work addresses the {\it holomorphy} of the simultaneously uniformizing map. Our construction is essentially equivalent to the one by McMullen; see Lemma \ref{lem:Equiv Btw Mating Maps}. The idea of using $\Hcal_d^{fm}$ to define a holomorphic structure on $\Bcal_d^{fm}$ is also mentioned in \cite[p.\@ 371]{McMullen08} and \cite{Milnor_HypComp}.

\subsection{Geometric structures on the spaces of dynamical systems}
The construction of holomorphic structure on $\Bcal_d^{fm}$ belongs to a broader program of
endowing the spaces of topological classes of dynamical systems with natural analytic
and geometric structures. In complex dynamics, Lyubich introduced a
complex analytic structure on the space of quadratic-like germs and showed
that hybrid classes form a holomorphic codimension-one foliation
\cite{Lyubich99}. Avila and Lyubich subsequently developed the relation
between hybrid classes of polynomial-like maps and
real-analytic expanding circle maps in their study of renormalization
\cite{AvilaLyubich11}. In real one-dimensional dynamics, an important
general problem is to understand how topological classes decompose into
finer rigidity classes; see, for example,
\cite{MartensPalmisanoWinckler18}. Smooth manifold structures on 
the spaces of topological classes and their infinitesimal deformation theory have been
studied by Grotta-Ragazzo and Smania
\cite{GrottaRagazzoSmaniaI,GrottaRagazzoSmaniaII}; see also
\cite{Smania25} for a recent overview. Natural analytic and Riemannian
structures arising from thermodynamic formalism have also been
investigated by Giulietti--Kloeckner--Lopes--Marcon in \cite{GiuliettiKloecknerLopesMarcon18}. Our construction
is different in that the complex structure on
$\mathcal B_d^{fm}$ is obtained through a simultaneous
uniformization of quasi-Blaschke products and is then used to study
McMullen's Weil--Petersson semi-norm.

\subsection{Summary of the proofs}

The complex structure on $\mathcal B_d^{fm}$ and the simultaneous
uniformization theorem are both obtained from Theorem~\ref{thm_simul}. We construct a
biholomorphism
\[
\Theta\from\mathcal H_d^{fm}\times
\overline{\mathcal H_d^{fm}}
\longrightarrow \mathcal{QB}_d^{fm}
\]
by quasiconformal surgery. Given $[f],[g]\in\mathcal H_d^{fm}$, we combine their invariant Beltrami
differentials and solve the resulting Beltrami equation. The rational map $\Theta([f],[g])$ obtained in this way is a quasi-Blaschke product whose restrictions to its two Fatou components are uniformized to the polynomial dynamics $f$ and $g$, respectively.
The map $\Theta$ sends the diagonal to $\mathcal B_d^{fm}\subset \mathcal{QB}_d^{fm}$, and hence it induces a natural diffeomorphism between
$\mathcal H_d^{fm}$ and
$\mathcal B_d^{fm}$. We use this identification to transport the
complex structure of $\mathcal H_d^{fm}$ to
$\mathcal B_d^{fm}$, from which Theorem 4 follows.

The main point is to prove the holomorphic dependence of the resulting
rational map on the parameters. This is analogous to Bers' proof of
simultaneous uniformization, but there is an additional difficulty in the
quasi-Blaschke setting. For quasi-Fuchsian groups, a normalized family of
M\"obius transformations is holomorphic once the images of three points
depend holomorphically on the parameter; see the application of \cite[Proposition 4.8.19]{Hub_vol1} in the proof of \cite[Proposition 6.12.4]{Hub_vol1}. However, for quasi-Blaschke products, since 
$\dim_{\mathbb C}\mathcal{QB}_d^{fm}=2d-2$, three points are not
enough to recover the resulting rational map. We normalize the maps so
that $0$ and $\infty$ are attracting fixed points and $1$ is a repelling
fixed point. The remaining coefficients are then determined by the
$d-1$ fixed points on the Julia set together with the $d$ preimages of
$1$. Since $1$ occurs in both collections, these give $2d-2$ holomorphically
moving points. This proves separate holomorphy, and Hartogs' theorem gives
joint holomorphy.

We prove Theorems \ref{main theorem 2} in Section \ref{sec_4.4}. To prove Theorem \ref{main theorem 2}-(1), suppose that there exists a non-zero tangent vector $\vec{v}\in T_{[f]}\Bcal^{fm}_d$ such that $||\vec{v}||_{WP}=0$. 
Using the relation with the infinitesimal marked multiplier rigidity (Lemma \ref{lem_degen}) and the holomorphic index formula (Equation \eqref{eqn:HoloIndex}), we obtain three possible cases in Lemma \ref{lem_cases}. Then, using the holomorphic structure of $\Bcal_d^{fm}$, the quasiconformal surgeries from Theorem \ref{main theorem 1}, and an idea due to Oleg Ivrii (Proposition \ref{prop_Ivrii_trick}), we show that one case cannot occur. The other two cases reduce to the situation where $\vec{v}\in T\mathcal{A}^{fm}_d$.
Theorem \ref{main theorem 2}-(2) follows from the fact that $\int_0^1 |\gamma'(t)|\, dt > 0$ for every $C^1$-curve $\gamma$; see Lemma \ref{p:non-deg-analytic-paths}.

\subsection*{Acknowledgement} The authors thank Oleg Ivrii for suggesting Proposition \ref{prop_Ivrii_trick}, Kostiantyn Drach for a helpful conversation, and Curtis McMullen for valuable comments.
The authors would also like to thank the referees for a careful reading and useful suggestions.
HL was supported in part by AMS-Simons travel grant, a KIAS Individual Grant (HP104101) via the June E Huh Center for Mathematical Challenges at Korea Institute for
Advanced Study, and Sanghyun Kim's 
Mid-Career Researcher Program (RS-2023-00278510) through the National Research Foundation funded by the government of Korea.

\section{Preliminaries} \label{sec_Prel}
In this section, we collect preliminary results that we will need for the rest of the paper. 
In Section \ref{sec_defQBd}, we review the definitions of fixed-point-marked (quasi)-Blaschke products and their associated spaces, following \cite{HLP23}.
In Section \ref{sec_2.2}, we record McMullen's construction of {\it matings} of two Blaschke products in \cite{McM_simul}. In Section \ref{sec_Curt's_metric}, we review McMullen's construction of the {\it Weil--Petersson semi-norm} on the space of Blaschke products in \cite{McMullen08}.

\subsection{The space of fixed-point-marked quasi-Blaschke products}\label{sec_defQBd}
In this section, we recall the space $\QB_d^{fm}$ of conjugacy classes of degree $d \ge 2$ quasi-Blaschke products with marked fixed points, which is rigorously defined in \cite{HLP23}. We also refer the reader to \cite{Milnor_HypComp} for a more general account of the hyperbolic components of fixed-point-marked rational maps.

For $d\ge 2$, any degree-$d$ rational map $f$ has $(d+1)$ fixed points $x_1,x_2,\dots,x_{d+1}$ counted with multiplicity. A rational map $f$ together with an ordered $(d+1)$-tuple of its fixed points $(f; x_1,x_2,\dots,x_{d+1})$ is called a {\it rational map with marked fixed points} or a {\it fixed-point-marked rational map}. For simplicity, we sometimes omit $x_i$'s and say that $f$ is a fixed-point-marked rational map when the marking of the fixed points are understood or inessential in the context. Denote by ${\rm Fix}(f)$ the unordered set of $(d+1)$ fixed points counted with multiplicity.

We define the space ${\rm Rat}_d^{fm}$ of degree-$d$ fixed-point-marked rational maps by
\[
    {\rm Rat}_d^{fm}:= \left\{(f;x_1,x_2,\dots,x_{d+1})\in {\rm Rat}_d \times \left(\mathbb{P}^{1}\right)^{d+1}~|~\{x_1,x_2,\dots,x_{d+1}\}={\rm Fix}(f)\right\}.
\]
By \cite[Lemma 9.2]{Milnor_HypComp}, $\Rat_d^{fm}$ is a smooth manifold. We remark that having $d+1$ distinct fixed points is equivalent to having no fixed points with multiplier $1$.

Recall that a rational map $f$ is {\it hyperbolic} if every critical point of $f$ lies in an attracting basin. We say that a fixed-point-marked rational map $(f;x_1,x_2,\dots,x_{d+1})$ is {\it hyperbolic} if $f$ is a hyperbolic rational map. A hyperbolic component in ${\rm Rat}_d^{fm}$ is a connected component of the subset of hyperbolic fixed-point-marked rational maps.

M\"obius transformations $\phi \in \PSL(2,\mathbb{C})$ act on ${\rm Rat}^{fm}_d$ by
\[
    \phi\cdot (f;x_1,x_2,\dots,x_{d+1}):= (\phi \circ f \circ \phi^{-1}; \phi(x_1),\phi(x_2),\dots,\phi(x_{d+1})),
\]
such that the action is free on the complement of the locus consisting of rational maps having less than three fixed points. In particular, the action is free on the set of hyperbolic fixed-point-marked rational maps. We denote by ${\rm rat}_d^{fm}$ the quotient of ${\rm Rat}_d^{fm}$ by the $\PSL(2,\mathbb{C})$-action and call it the space of conjugacy classes of degree-$d$ rational maps with marked fixed points. We refer the reader to \cite[Section 9]{Milnor_HypComp} for details on fixed-point-marked rational maps.

\begin{defn}[Quasi-Blaschke products] \label{def_QB}
A hyperbolic rational map $f$ is said to be a {\it quasi-Blaschke product} if its Julia set $\mathcal{J}(f)$ is a quasi-circle, and $f$ fixes each of the two Fatou components.  
\end{defn}

\begin{defn}[Space of quasi-Blaschke products]
For $d\ge2$, we define $\widetilde{\QB}^{fm}_d$ as the connected component of fixed-point-marked degree-$d$ quasi-Blaschke products containing $(z^d;0,\infty,1,\zeta,\zeta^2,\dots,\zeta^{d-2})$, where $\zeta=e^{2\pi i\frac{1}{d-1}}$.
We define the moduli space of degree-$d$ fixed-point-marked quasi-Blaschke products $\QB^{fm}_d$ by $\QB^{fm}_d:=\widetilde{\QB}^{fm}_{d}/\PSL(2,\mathbb{C})$.
\end{defn}

The spaces $\widetilde{\QB}^{fm}_d$ and $\QB^{fm}_d$ are complex manifolds, and they are hyperbolic components of ${\rm Rat}_d^{fm}$ and ${\rm rat}_d^{fm}$, respectively, \cite[Lemma 3.5]{HLP23}.

\begin{defn}[Space of Blaschke products]
Define the moduli space of degree-$d$ fixed-point-marked Blaschke products $\Bcal_d^{fm}$ as a subspace of $\QB_d^{fm}$ by
\[
    \Bcal_d^{fm} : =\left\{[f] \in \QB_d^{fm}: f \text{ is a Blaschke product} \right\}.
\]
\end{defn}

Recall that for $d\ge 2$, a degree-$d$ {\it Blaschke product} is a rational map $f \from \mathbb{P}^1 \to \mathbb{P}^1$ of the form
    \begin{equation}\label{eq_def_blaschke}
        f(z) = e^{2\pi i\theta}\prod_{i=1}^{d} \frac{z-a_i}{1-\overline{a_i}z},
    \end{equation}
where $(a_1,\ldots, a_d)\in \Delta^{d}$ and $\theta\in \mathbb{R}/\mathbb{Z}$. Here, $\Delta=\{z\in \mathbb{P}^{1}:|z|<1\}$.  Its Julia set $\Julia(f)$ is the unit circle $\partial \Delta$. It is easy to show that the space $\Bcal^{fm}_d$, which is defined as a set of ${\rm PSL}(2,\mathbb{C})$-conjugacy classes of certain fixed-point-marked rational maps, is equivalent to the space of ${\rm Aut}_{\mathbb{C}}(\Delta)$-conjugacy classes of Blaschke products, which are of the form in Equation \eqref{eq_def_blaschke}, with marked fixed points.

\begin{defn}[Standard representatives]\label{defn:StandardRep}
    For any $[(f;x_1,x_2,\dots,x_{d+1})]$ in $\Bcal_d^{fm}$ or $\QB_d^{fm}$, we say that $(f;x_1,x_2,\dots,x_{d+1})$ is a {\it standard representative} of its $\PSL(2,\mathbb{C})$-conjugacy class if $f$ is a Blaschke product, $x_1=0,x_2=\infty$, and $x_3=1$. 
\end{defn}

It is easy to show that for any $[f]\in \Bcal^{fm}_d$, if $f$ is the standard representative of $[f]$, then $f$ has the form
\[
    f(z) = \left(\prod_{i=1}^{d-1}\frac{1-\overline{a_i}}{1-a_i}\right)z\prod_{i=1}^{d-1} \frac{z-a_i}{1-\overline{a_i}z},
\]
together with the marking of its fixed points, where $a_i$'s are in $\Delta$.

\subsection*{Multiplier functions}
Consider $(z^d;0,\infty,1,\zeta^1,\dots,\zeta^{d-2})$, a fixed-point-marked monomial, where $\zeta=e^{2\pi i/(d-1)}$. For any standard representative of the fixed-point-marked Blaschke product $(f;0,\infty,1,x_4,\dots,x_{d+1})$, there exists a unique homeomorphism $\phi_f \colon S^1 \to S^1$ conjugating $(S^1,z\mapsto z^d)$ to $(S^1,f)$ such that $\phi(\zeta^k)=x_{k+3}$ for $k\in\{1,2,\dots,d-2\}$.

Suppose that $C = \{x = x^{d^n},x^{d},x^{d^2},\ldots,x^{d^{(n-1)}}\}$ is an $n$-cycle with $n\ge 1$ of the map $z \mapsto z^d$ on the unit circle. Then, $\phi_f(C)$ is an $n$-cycle of $f$ whose multiplier $\lambda_C(f)$ is
$$\lambda_C(f) := (f^n)'(\phi_f(x)).$$
Since $\phi_f$ is uniquely determined, we call $\phi_f(C)$ the {\it $n$-cycle of $C$ for $f$}  and $\lambda_C(f):=\lambda_{\phi_f(C)}(f)$ the {\it multiplier of the $n$-cycle $C$ for $f$}. Since the multipliers are invariant under smooth conjugacies, we can also consider $\lambda_C$ as a function $\lambda_C\from \Bcal_d^{fm}\to \Rbb$, sending $[f]\mapsto \lambda_C([f]):=\lambda_C(f)$.

We recall that for any $[f]\in \Bcal_d^{fm}$ and any $n$-cycle $C$ in $S^1$, the cycle $C$ is {\it repelling} for $[f]$ , i.e., $|\lambda_C([f])|>1$.

\subsection*{Holomorphic index formula} 
Recall that a fixed point $x$ of a rational map $f$ is \emph{attracting}
if $|f'(x)|<1$, and its multiplier is $f'(x)$. Any Blaschke product $f$ has two attracting fixed points: one, denoted by $x$, lies inside $\Delta$ and the other $1/\overline{x}$ inside $1/\Delta$. Note that $f$ is symmetric under the anti-holomorphic reflection $1/\overline{z}$, i.e., $f\circ(1/\overline{z})=(1/\overline{z})\circ f$. Hence, for the attracting multiplier $\lambda_{att}(f)$ at $x$, the attracting multiplier at $1/\overline{x}$ is equal to $\overline{\lambda_{att}(f)}$. 

We note that for any $n\ge1$, $f^{\circ n}$ is also a Blaschke product with $\lambda_{att}(f^{\circ n})=\lambda_{att}(f)^n$. Hence, $1-\lambda_{att}(f)^n \neq 0$ for any $n\ge1$. For any $n\ge 1$, by applying the holomorphic index formula \cite[Chapter 12]{Milnor06} to $f^{\circ n}$, we obtain
\begin{equation}\label{eqn:HoloIndex}
    \sum_{m|n}\sum_{~C {\rm\,is\,a\,repelling}\,m{\textup -}{\rm cycle}} \frac{m}{\lambda_C(f)-1}=\frac{1-|\lambda_{att}(f)^n|^2}{|1-\lambda_{att}(f)^n|^2}.
\end{equation}

\subsection{Matings of Blaschke products}\label{sec_2.2}
This section summarizes the results and proofs in \cite{McM_simul}.
Let $\mathcal{C}$ be a Jordan curve in $\mathbb P^1$, and let $U$ and $V$ be the two components of $\mathbb P^1 \setminus \mathcal{C}$. Then there exist biholomorphisms $g_1 \colon \Delta \to U$ and $g_2 \colon 1/\Delta \to V$, 
which continuously extend to the boundary circle $S^1=\partial{\Delta}=\partial(1/\Delta)$. The {\it conformal welding} of $\mathcal{C}$ is the circle homeomorphism $h_{\mathcal{C}}:= g_2^{-1}\circ g_1|_{S^1} \colon S^1 \to S^1$, which is unique up to pre- and post-compositions with ${\rm Aut}(\Delta)$.

\begin{lem} \label{lem_an_thm3}
For any quasisymmetry $h \colon S^1 \to S^1$, there exists a quasiconformal map $\Psi \colon \mathbb{P}^1 \to \mathbb{P}^1$ such that $h = h_{\mathcal{C}}$ where
$\mathcal{C} = \Psi(S^1)$ is a quasi-circle.
\end{lem}
\begin{proof}
The statement is proven in \cite[Proof of Theorem 3]{McM_simul} as follows. We first observe that $h$ extends to a quasiconformal map $h \colon \overline{\Delta} \to \overline{\Delta}$ of the closed unit disk. Let $\mu$ be the complex dilatation of $h^{-1}$ and extend $\mu$ to be 0 on $1/\Delta$. It then follows from the measurable Riemann mapping theorem that there exists a quasiconformal map $\Psi \from \mathbb{P}^1 \to \mathbb{P}^1$ such that $\partial_{\bar{z}}\Psi/\partial_z \Psi = \mu$ almost everywhere. Then $\Psi(S^1)$ is a quasi-circle. See Figure \ref{fig:McMSUnif}. To see that $h_{\Psi(S^1)}$ is equal to $h$, we note that by construction, $\Psi$ maps $\mathbb{P}^1\setminus \Delta$ conformally onto one of the component of $\mathbb{P}^1\setminus \Psi(S^1)$ and $\Psi \circ h$ maps $\Delta$ conformally onto the other. Therefore the boundary correspondence is $(\Psi)^{-1}(\Psi \circ h) = h$.
\end{proof}
\begin{figure}[h!]
	\centering
	\def\svgwidth{0.8\textwidth}
	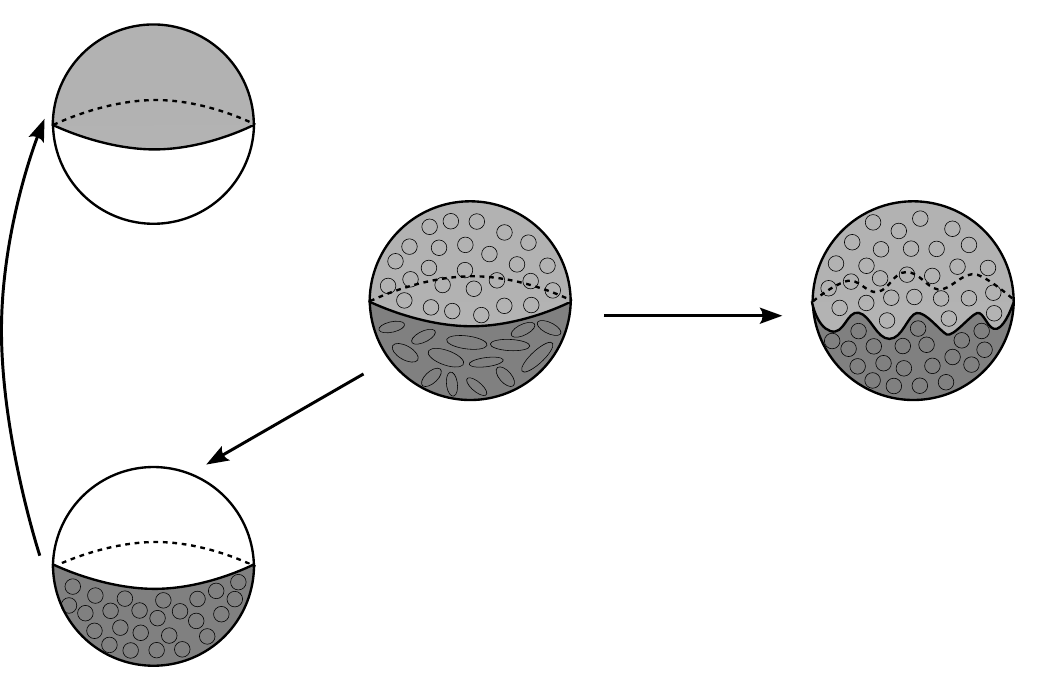
	\caption{${\rm Mate}\from \Bcal_d^{fm}\times \Bcal_d^{fm}\to \QB_d^{fm}$}
    \label{fig:McMSUnif}
\end{figure}  

\begin{rem}[\cite{McM_simul}]\label{rem_bijection}
Conformal welding gives a bijection
between the following two sets of equivalence classes:
$$\left\{\text{quasisymmetries $h\from S^1 \to S^1$}\right\}/\left\{h \sim A\circ h\circ B : A,B\in {\rm Aut}(\Delta)\right\} $$ and 
$$\left\{\text{quasi-circles $\mathcal{C}$}\right\}/\left\{\mathcal{C} \sim A(\mathcal{C}): A \in {\rm Aut}(\mathbb{P}^1)\right\}.$$ The correspondence sends the class of a quasisymmetry
$h$ to the M\"obius-equivalence class of its welded quasicircle $C_h$.
Conversely, if $C$ is a quasicircle and
$g_1\colon \Delta\to U$, $g_2\colon \Delta^*\to V$ are conformal maps onto the two
components of $\mathbb P^1\setminus C$, then the inverse correspondence
is given by the welding homeomorphism
$h_C=g_2^{-1}\circ g_1|_{S^1}$.
The resulting equivalence classes are independent of the choices of
$g_1$ and $g_2$.
\end{rem}

Now let us discuss the {\it mating} of two Blaschke products with marked fixed points. Suppose $[f_1],[f_2]\in \Bcal_d^{fm}$ such that $f_1$ and $f_2$ are standard representatives. Then there is a quasisymmetric homeomorphism $h \colon S^1 \to S^1$ such that $h \circ f_1 = f_2\circ h$ on $S^1$ with $h(1)=1$. Consider $\Psi$ defined in Lemma \ref{lem_an_thm3}. See also Figure \ref{fig:McMSUnif}. We can normalize $\Psi$ such that it fixes $0,1$, and $\infty$.

Consider the quasi-circle $\Psi(S^1)$ in $\mathbb{P}^1$. Denote by $U$ and $V$ the components of $\mathbb{P}^1\setminus \Psi(S^1)$ such that $0\in U$ and $\infty \in V$. Define a map $F\from \mathbb{P}^1\setminus \Psi(S^1)\to \mathbb{P}^1 \setminus \Psi(S^1)$ by
\begin{equation*}
    F(z)=\left\{
    \begin{array}{cc}
        \left((h^{-1} \circ \Psi^{-1})^{-1} \circ  f_1 \circ (h^{-1} \circ \Psi^{-1})\right)(z) & z\in U\vspace{5pt}\\
        \left(\Psi \circ  f_2 \circ (\Psi^{-1})\right)(z) & z\in V.
    \end{array}
    \right.
\end{equation*}
It is easy to check that $F$ continuously extends to $\mathbb{P}^1$. Since $\Psi(S^1)$ is a quasi-circle, it has measure zero. Therefore, $F$ is a quasi-regular map that is holomorphic almost everywhere, which implies $F$ is a rational map. The Julia set $\Julia(F)$ of $F$ is $\Psi(S^1)$. We emphasize that $\Psi\circ h|_\Delta$ holomorphically conjugates $f_1|_\Delta$ to $F|_U$, and $\Psi \circ {\rm id}|_{1/\Delta}$ holomorphically conjugates $f_2|_{1/\Delta}$ to $F|_V$. We call $F$ the {\it mating} of $f_1$ and $f_2$, and we denote $[F]\in \QB_d^{fm}$ by ${\rm Mate}([f_1],[f_2])$. It gives rises to a map ${\rm Mate}\from \Bcal_d^{fm}\times \Bcal_d^{fm}\to \QB_d^{fm}$.

\begin{lem}\label{lem_smooth_fam}
Let $([f_t])_{t \in (-1,1)}$ be a smooth family in $\Bcal_d^{fm}$ and $[g] \in \Bcal_d^{fm}$. Then ${\rm Mate}([f_t],[g])$ is smooth in $t$.
\end{lem}
\begin{proof}
Suppose that $([f_t])_{t \in (-1,1)}$ is a smooth family in $\Bcal_d^{fm}$ and $[g] \in \Bcal_d^{fm}$ such that $f_t$ and $g$ are standard representatives. Let $F_t$ be the standard representative of ${\rm Mate}([f_t],[g])$ for $t\in (-1,1)$. We want to show that $F_t$ is smooth in $t$.

Let us first show that $h_t\from S^1\to S^1$ is smooth in $t$, where $h_t \colon \Julia(g)\to \Julia(f_t)$ is the quasisymmetry conjugating the dynamics of $g$ and $f_t$ as well as their markings of the fixed points. Denote by $\widetilde{f}_t \colon \mathbb{R} \to \mathbb{R}$ and $\widetilde{g}\colon \mathbb{R} \to \mathbb{R}$ the lifts of $f_t \colon S^1 \to S^1$ and $g\colon S^1 \to S^1$ to the universal cover $\mathbb R$ of $S^1$. Normalize them so that $\widetilde{f}_t(0) = \widetilde{g}(0)=0$. By \cite[Proof of Lemma 4]{McM_simul}, the lift $\widetilde{h}_t \colon \mathbb{R} \to \mathbb{R}$ of the conjugating map $h_t \colon \Julia(g) \to \Julia(f_t)$ is given by
\begin{equation}\label{eq_def_htilde}
\widetilde{h}_t(x) := \lim_{n \to \infty} \widetilde{f}_t^{-n} \circ \widetilde{g}^n(x), \qquad x \in \mathbb{R}.
\end{equation}
It follows from \eqref{eq_def_htilde} that $\widetilde{h}_t \colon \mathbb{R} \to \mathbb{R}$ is smooth in $t$. Therefore $h_t \colon S^1 \to S^1$ is smooth in $t$. 

Next, we extend $(h_t\from S^1\to S^1)$ to a smooth family of quasiconformal maps $(h_t\from \Delta\to \Delta)$. Let $\widehat{h}_t :=\alpha \circ h_t \circ \alpha^{-1} \colon \mathbb{R} \to \mathbb{R}$ for a M\"obius transformation $\alpha \colon \Delta \to \mathbb{H}$.
We use the construction in \cite[Proof of Theorem 4.9.5]{Hub_vol1} to obtain a quasiconformal extension $\widehat{h}_t \colon \mathbb{H} \to \mathbb{H}$, which is smooth in $t$. Therefore $h_t =\alpha^{-1} \circ \widehat{h}_t \circ \alpha \colon \overline{\Delta} \to \overline{\Delta}$ is smooth in $t$. Then, by the construction of the map ${\rm Mate}$, we obtain the smooth dependence of ${\rm Mate}([f_t],[g])$ on $t$.
\end{proof}

\begin{rem}
At this point, we do not know whether the map ${\rm Mate}\from \Bcal_d^{fm} \times \Bcal_d^{fm} \to \QB_d^{fm}$ is smooth. Coordinate-wise differentiability does not, in general, imply differentiability with respect to real variables. However, for complex variables this implication does hold by Hartogs’ theorem (see for example \cite[Section 2.4]{KrantzSCV}), which is used in the proofs of Theorems \ref{main theorem 1} and \ref{thm_simul}.
\end{rem}

\subsection{McMullen's Weil--Petersson semi-norm} \label{sec_Curt's_metric}

Suppose that $([f_t])_{t \in (-1,1)}$ is a smooth path in $\mathcal{B}_d^{fm}$ such that $f_t$ is the standard representative of $[f_t]$.
For each $t \in (-1,1)$, denote by $[F_t] := {\rm Mate}([f_0],[f_t])$ such that $F_t$ is the standard representative of $[F_t]$. Then, $F_t \colon \mathbb{P}^1 \to \mathbb{P}^1$ is a smooth family of rational maps by Lemma \ref{lem_smooth_fam} with $F_0=f_0$.

There is a unique smooth family of conformal maps
$$H_t\from \Delta \to \mathbb{P}^1$$
satisfying $H_t \circ F_0=F_t\circ H_t$ and $H_0(z) = z$. Each $H_t$ extends to a quasiconformal map on $\mathbb{P}^1$, sending $S^1$ to the Julia set $\mathcal{J}(F_t)$ of $F_t$.

For any tangent vector $\Vec{v} := d/dt|_{t=0}[f_t] \in T_{[f_0]}\Bcal_d^{fm}$, we define a holomorphic vector field $\eta_{\Vec{v}}$ on $\Delta$ by
\begin{equation}\label{def_eta_vf}
    \eta_{\Vec{v}}(z) := \frac{d}{dt}\bigg|_{t=0}H_t(z).
\end{equation}
McMullen defined the {\it Weil--Petersson semi-norm} $||\cdot ||_{WP}$ on $\Bcal_d^{fm}$ by
\begin{equation}\label{eq_def_metric}
||\Vec{v}||_{WP} := \lim_{r\to1} \frac{1}{4\pi|\log(1-r)|}\int_{|z|=r} |\eta_{\Vec{v}}'(z)|^2|dz|;
\end{equation}
see \cite[Theorem 1.7]{McMullen08}.

We denote by $||\cdot||_{\mathcal{P}}$ the {\it pressure semi-norm} in $\Bcal_d^{fm}$ defined in \cite[Section 2.2]{HLP23}.

\begin{thm}[{\cite[Theorems 1.6, 1.7, and 2.6]{McMullen08}}]\label{thm:McMullen4}
Let $([f_t])_{t\in(-1,1)}$ be a smooth path in $\Bcal_d^{fm}$. For $\Vec{v} := \left.\frac{d}{dt}\right|_{t=0}[f_t] \in T_{[f_0]}\Bcal_d^{fm}$, we have
\[
    ||\Vec{v}||_{WP} = \frac{1}{2} || \Vec{v}||_\Pcal.
\]
\end{thm}

\begin{lem}\label{lem_degen}
Let $([f_t])_{t\in(-1,1)}$ be a smooth path in $\Bcal_d^{fm}$  such that $|| \left.\frac{d}{d t}\right|_{t=0} [f_{t}] ||_{WP}=0.$ Then, for any repelling $n$-cycle $C$ with $n\ge1$, we have
$$\left.\frac{d}{d t}\right|_{t=0} \lambda_C([f_t])=0.$$
\end{lem}
\begin{proof}
Suppose $|| \left.\frac{d}{d t}\right|_{t=0} [f_t] ||_{WP}=0.$ By Theorem \ref{thm:McMullen4}, $|| \left.\frac{d}{d t}\right|_{t=0} [f_t]||_{\mathcal{P}} = 0$. The conclusion then follows from \cite[Proposition 2.9]{HLP23} or \cite[Corollary 4.3]{HeNie_MetricHypComp}.
\end{proof}

Lemma \ref{lem_degen} is the only statement in this article that requires the pressure semi-norm $||\cdot||_{\mathcal{P}}$. Therefore, we omit further details of $||\cdot||_{\mathcal{P}}$ and refer the reader to \cite{HeNie_MetricHypComp,HLP23} for a comprehensive discussion.

\section{Simultaneous uniformization of quasi-Blaschke products} \label{sec_pf_1.1}
The goal of this section is to prove Theorem \ref{main theorem 1}---holomorphic simultaneous uniformization for Blaschke products. 
To this end, in Section \ref{sec_simul_central}, we first uniformize any quasi-Blaschke product into a pair of polynomials in the central component $\mathcal{H}_d^{fm}$ of ${\rm poly}_d^{fm}$; see Theorem \ref{thm_simul}. In Section \ref{sec_holo_simul_BD}, using Theorem \ref{thm_simul}, we define a complex structure on $\Bcal^{fm}_d$ and prove Theorem \ref{main theorem 1}. In Lemma \ref{lem:Equiv Btw Mating Maps}, we also show that our mating map $\mathcal{U}\from \Bcal^{fm}_d \times \overline{\Bcal^{fm}_d} \to \QB^{fm}_d$ coincides with the map ${\rm Mate}$ by McMullen (see Section \ref{sec_2.2}).

\subsection{Simultaneous uniformization with central components of polynomials}\label{sec_simul_central}
Let ${\rm poly}_d^{fm}$ be the subset of ${\rm rat}_d^{fm}$ consisting of equivalence classes of fixed-point-marked degree-$d$ polynomials.
We denote by $\mathcal{H}_d^{fm}$ the central hyperbolic component in ${\rm poly}_d^{fm}$, i.e., the hyperbolic component in ${\rm poly}_d^{fm}$ containing $[(z^d;0,\infty,1,\zeta,\dots,\zeta^{d-2})]$, where $\zeta=e^{2\pi i/(d-1)}$.
We note that $\mathcal{H}_d^{fm}$ is a complex manifold; see \cite[p.\@ 16]{HLP23} or \cite[Lemma 9.2]{Milnor_HypComp}.

For any complex manifold $X$, we denote by $\overline{X}$ the complex manifold obtained by conjugating the complex structure of $X$. We denote by ${\rm Diag}$ the {\it diagonal} of $\mathcal{H}_d^{fm} \times \overline{\mathcal{H}_d^{fm}}$; namely,
$${\rm Diag} :=\{([f],[f]):[f]\in\mathcal{H}_d^{fm}\} \subset \mathcal{H}_d^{fm} \times \overline{\mathcal{H}_d^{fm}}.$$ 

\begin{rem}
    Let $X$ be a Riemann surface of genus $g$. In Bers' simultaneous uniformization $ \mathcal{T}(X)\times \mathcal{T}(\overline{X}) \to \mathcal{QF}_g$, the skew-diagonal $(Y,\overline{Y})$ maps to the Fuchsian locus. Here, if $Y=\mathbb{H}/\Gamma$ with $\Gamma \le \mathrm{PSL}(2,\mathbb{R})$, then $\overline{Y} = \overline{\mathbb{H}} / \overline{\Gamma}$. In complex dynamics, however, Blaschke products are already symmetric under conjugation along the unit circle; that is, $B(1/\overline{z}) = 1/\overline{B(z)}$.
    Therefore, the diagonal, rather than the skew-diagonal, of  $\mathcal{H}_d^{fm} \times \overline{\mathcal{H}_d^{fm}}$ maps to the Blaschke locus.
\end{rem}

\begin{thm}[Simultaneous uniformization with central components of polynomials]\label{thm_simul}
    There is a biholomorphism
    \[
        \Theta \from \mathcal{H}_d^{fm} \times \overline{\mathcal{H}_d^{fm}} \to \mathcal{QB}_d^{fm}
    \]
    such that the following hold.
    \begin{enumerate}
        \item Suppose that $F$ is a rational map in the equivalence class $\Theta([f],[g])$ and $U$ and $V$ are the two Fatou components of $F$. Then, $F|_U$ is holomorphically conjugate to $f|_{{\rm int}(K(f))}$ and $F|_V$ is anti-holomorphically conjugate to $g|_{{\rm int}(K(g))}$, where $K(f)$ and $K(g)$ are the filled Julia sets of $f$ and $g$, respectively.
        \item The restriction $\Theta|_{\rm Diag}$ to the diagonal is a diffeomorphism between ${\rm Diag}$ and $\mathcal{B}_d^{fm}$.
    \end{enumerate}
\end{thm}
\begin{proof}
{\it (Step 1) We first define the map $\Theta$.} Let $[f],[g] \in \mathcal{H}_d^{fm}$ so that $f$ and $g$ are their standard representatives. There exist quasiconformal homeomorphisms $\phi_f,\phi_g\from \mathbb{P}^1 \to \mathbb{P}^1$ such that
\begin{enumerate}
    \item $\phi_f$ and $\phi_g$ conjugate the dynamics of $z^d$ near the circle $\partial \Delta$ to the dynamics of $f$ and $g$ near the Julia sets $\Julia(f)$ and $\Julia(g)$, respectively,
    \item $\phi_f|_{1/\Delta}$  and $\phi_g|_{1/\Delta}$ holomorphically conjugate the monomial $z^d\from 1/\Delta\to 1/\Delta$ to $f\from {K(f)}^c  \to {K(f)}^c$ and to $g\from {K(g)}^c  \to {K(g)}^c$, respectively, and
    \item $\phi_f$ and $\phi_g$ fix three points: $0$ and $\infty$ (attracting or super-attracting fixed points), and $1$ (a repelling fixed point).    
\end{enumerate}

Define $\mu_f:= \partial_{\bar{z}}\phi_f/\partial_{z} \phi_f$ and $\mu_g:= \partial_{\bar{z}}\phi_g/\partial_{z} \phi_g$.
We note that $\mu_f =\mu_g\equiv 0$ in $1/\Delta$. We define a mated Beltrami differential $\mu_{f *g}$ by 
    \[
        \mu_{f*g}(z):=\left\{
        \begin{array}{cc}
            \mu_f(z) & z \in \Delta\\
            \overline{\mu_g(\frac{1}{\bar{z}})}\cdot \frac{z^2}{\bar{z}^2} & z \in 1/\Delta
        \end{array}
        \right.,
    \]
where $\mu_{f*g}$ on $1/\Delta$ is defined as the pullback of $\mu_g|_\Delta$ via $z\mapsto 1/\bar{z}$.
Let $\phi_{f*g} \from \mathbb{P}^1 \to \mathbb{P}^1$ be the unique quasiconformal map, obtained by the measurable Riemann mapping theorem, that solves $\mu_{f*g}=\partial_{\bar{z}} \phi_{f*g}/\partial_z \phi_{f*g}$ and fixes $0,1,$ and $\infty$.
Let $U:=\phi_{f*g}(\Delta)$ and $V:= \phi_{f*g}(1/\Delta)$. Define $\bar{g}$ and $\phi_{\bar{g}}$ as the conjugates of $g$ and $\phi_g$ by the map $z\mapsto 1/\bar{z}$, respectively. Then, we have
\[
    \mu_{\bar{g}}(z):=\partial_{\bar{z}}\phi_{\bar{g}}/\partial_z\phi_{\bar{g}}=\overline{\mu_g(1/\bar{z})}\frac{z^2}{\bar{z}^2} ~\frac{d\bar{z}}{dz}.
\]
We define $F \from \mathbb{P}^1 \to \mathbb{P}^1$ by
    \[
        F(z):=\left\{
        \begin{array}{cc}
            (\phi_f \circ \phi_{f*g}^{-1})^{-1} \circ f \circ (\phi_f \circ \phi_{f*g}^{-1})(z) & z\in U\\
            (\phi_{\bar{g}} \circ \phi_{f*g}^{-1})^{-1} \circ \bar{g} \circ (\phi_{\bar{g}} \circ \phi_{f*g}^{-1})(z) & z\in V
        \end{array}.
        \right.
    \]
\begin{figure}[h!]
	\centering
	\def\svgwidth{0.8\textwidth}
	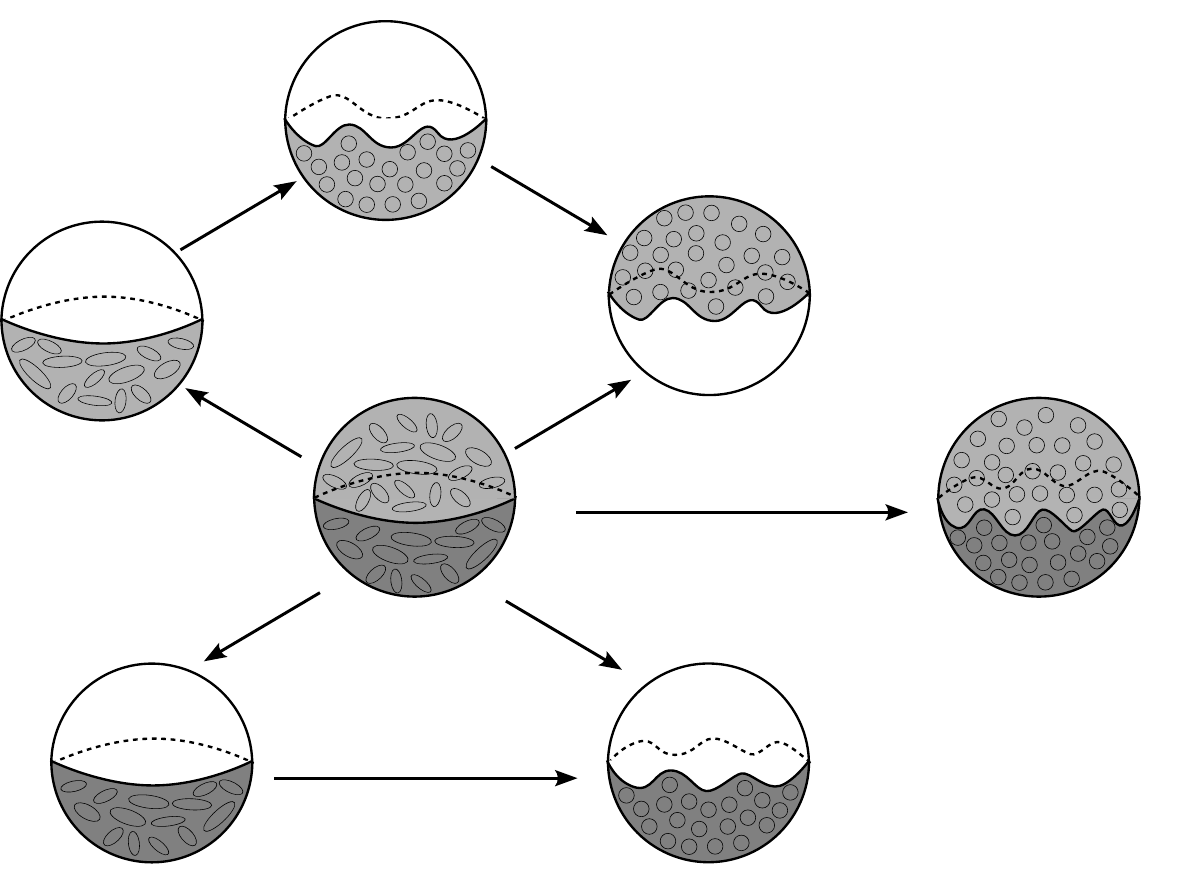
	\caption{Construction of the map $\Theta$.}
    \label{fig_1}
\end{figure}
We note that $F$ defined above continuously extend to the boundary $\partial U=\partial V$ because $\phi_f$ and $\phi_g$ also conjugate the dynamics $z^d \from S^1 \to S^1$ to $f \from \Julia(f) \to \Julia(f)$ and $g \from \Julia(g) \to \Julia(g)$ respectively.

Next, we show that $F\from \mathbb{P}^1 \to \mathbb{P}^1$ is holomorphic on $(U \cup V)$. Since $\mu_f=\mu_{f*g}$ on $\Delta$, $\phi_f$ and $\phi_{f*g}$ send the same field of ellipses on $\Delta$ to the fields of circles in $\phi_f(\Delta)={\rm int}(K(f))$ and $\phi_{f*g}(\Delta)=U$, respectively. Hence, $(\phi_f\circ \phi^{-1}_{f*g})$ sends a field of circles on $U$ to a field of circles on ${\rm int}(K(f))$. See Figure \ref{fig_1}. Hence, $\phi_f\circ \phi^{-1}_{f*g}|_U\from U \to {\rm int}(K(f))$ is a biholomorphism, and so is $\phi_f\circ \phi^{-1}_{f*g}|_V\from V \to {\rm int}(K(g))$ by the same argument.

Since $\phi_{f*g}(S^1)=(U \cup V)^c$ is a quasi-circle, it has Lebesgue measure zero. It follows that $F$ is a quasiregular map that is holomorphic almost everywhere. Therefore $F$ is a rational map.

Moreover, the Julia set $\Julia(F)$ of $F$ is the quasi-circle $\phi_{f*g}(\partial \Delta)$. 
By construction, $F$ is a quasi-Blaschke product such that $F|_U$ and $F|_V$ are holomorphically conjugate to $f|_{{\rm int}(K(f))}$ and anti-holomorphically conjugate to $g|_{{\rm int}(K(g))}$, respectively.

Let us define a marking of fixed points for $F$ as follows. By construction, $0$ and $\infty$ are attracting fixed points of $F$, and $1$ is a repelling fixed point. We define a marking of the fixed points of $F$ in such a way that $x_1=0$, $x_2=\infty$, $x_3=1$, and $x_3,x_4,\dots,x_{d+1}$ are in counter-clockwise order on $\Julia(F)$. Then $(F;x_1,x_2,\dots,x_{d+1})$ is the standard representative of $[(F;x_1,x_2,\dots,x_{d+1})]\in \QB_d^{fm}$.

We define $\Theta([f],[g]) := [(F;x_1,x_2,\dots,x_{d+1})]$.

{\it (Step 2) We show that $\Theta$ is holomorphic.} Since $\mathcal{H}_d^{fm}$ is a hyperbolic component, by the theory of holomorphic motions \cite{Lyu83typical, MSS_DynRatMap}, $\phi_f$ and $\phi_g$ have holomorphic dependence on $[f]$ and $[g]$ in $\Hcal^{fm}_d$. Moreover, by the analytic dependence of the solution on the Beltrami coefficient \cite[Theorem 4.7.4]{Hub_vol1}, $\phi_{f*g}$ depends holomorphically on $[f]$ and anti-holomorphically on $[g]$, where $[f]$ and $[g]$ are in $\Hcal^{fm}_d$.

Suppose that $(f_t)$ is a holomorphic family in $\mathcal{H}_d^{fm}$. Let $F_t$ be the standard representative of $\Theta([f_t],[g])$. Let $\phi_t$ denote the map $\phi_{f_t*g}$ constructed above. Then for any $z\in \mathbb{P}^1$, the map $t\mapsto \phi_t(z)$ is holomorphic. Moreover, $\phi_t \circ z^d = F_t \circ \phi_t$ on $S^1$.

If $z_0 \in S^1$ is a fixed point of $z^d$, then $\phi_t(z_0)$ is a fixed point of $F_t$. Also, if $z_0\in S^1$ is a $z^d$-preimage of $1\in S^1$, then $\phi_t(z_0)$ is a $F_t$-preimage of $1 \in \Julia(F_t)$. Hence the $d-1$ fixed points of $F_t$ on $\Julia(F_t)$ and the $d$ preimages of $1$ of $F_t$ change holomorphically with $t$. We use them to show $(F_t)$ is a holomorphic family as follows.

Since $F_t$ fixes $0$ and $\infty$, we can write $F_t(z)=\frac{p_t(z)}{q_t(z)}$ such that $p_t(z)$ is monic, $z~|~p_t(z)$, $\deg(p_t)=d$, and $\deg(q_t)\le d-1$. The $d-1$ fixed points on $\Julia(F_t)$ are the non-zero solutions of $p_t(z)-z q_t(z)=0$. The $d$ preimages of $1$ by $F_t$ are solutions of $p_t(z)-q_t(z)=0$. Since $p_t(z)-q_t(z)$ is monic, it is a holomorphic family of monic polynomials, say $(r_t(z))$. Also, $p_t(z)-z q_t(z)=a_t\cdot k_t(z)$, where $(k_t(z))$ is a holomorphic family of monic polynomials and $a_t\in \Cbb$. We have
\begin{align}\label{eqn:p_t and q_t}
\begin{split}
    p_t(z) & = z \frac{r_t(z)}{z-1}-a_t\frac{k_t(z)}{z-1}\\
    q_t(z) & = \frac{r_t(z)}{z-1}-a_t\frac{k_t(z)}{z-1}
\end{split},
\end{align}
where both $r_t(z)$ and $k_t(z)$ are divisible by $z-1$. By substituting $z$ with $0$ in \eqref{eqn:p_t and q_t} and using $p_t(0)=0$, we obtain that $t\mapsto a_t$ is holomorphic. Hence $(F_t)$ is a holomorphic family.

Therefore, we conclude that $([f],[g])\mapsto \Theta([f],[g])$ is holomorphic in $[f]$. Similarly, we can obtain the anti-holomorphic dependence on $[g]$. 

Hence $\Theta\from \Bcal^{fm}_d \times \overline{\Bcal^{fm}_d}\to \QB^{fm}_d$ is a holomorphic map. 

{\it (Step 3) We show that $\Theta$ is a bijective holomorphic map, and hence it is a biholomorphism.} To this end, we define a map $\Theta^{-1} \from \mathcal{QB}_d^{fm} \to \mathcal{H}_d^{fm} \times \overline{\mathcal{H}_d^{fm}}$ and show that it is the inverse of the map $\Theta$.

Suppose that $F$ is the standard representative of $[F]\in \mathcal{QB}_d^{fm}$.
There is a quasiconformal homeomorphism $\phi_F \from \mathbb{P}^1 \to \mathbb{P}^1$ that conjugates $z^d\from S^1 \to S^1$ with $F\from \Julia(F) \to \Julia(F)$. For $\mu_F:= \partial_{\bar{z}} \phi_F/ \partial_z \phi_F$, we define 
    \[
        \mu_1(z):=\left\{
        \begin{array}{cc}
            \mu_F(z) & z\in \Delta \\
            0 & z \in 1/\Delta
        \end{array}
        \right.
    \]
    and
    \[
        \mu_2(z):=\left\{
        \begin{array}{cc}
            \overline{\mu_F(\frac{1}{\bar{z}})}\frac{z^2}{\bar{z}^2} & z\in \Delta\\
            0 & z \in 1/\Delta 
        \end{array}
        \right..
    \]
For $ i\in \{1,2\}$, let $\phi_i \from \mathbb{P}^1 \to \mathbb{P}^1$ be quasiconformal homeomorphism, obtained from the measurable Riemann mapping theorem, which solves the Beltrami equation $\mu_i=\partial_{\bar{z}}\phi_i/\partial_z \phi_i$. We note that $\phi_i$ is a biholomorphism on $1/\Delta$.

For
$U_i:=\phi_i(\Delta)$ and $V_i:=\phi_i(1/\Delta)$, we define $f_i \colon \mathbb{P}^1 \to \mathbb{P}^1$ by
    \[
        f_i(z) := \left\{
        \begin{array}{cc}
            (\phi_F \circ \phi_i^{-1})^{-1} \circ F \circ (\phi_F \circ \phi_i^{-1})(z) & z\in U_i \\
            \phi_i \circ z^d \circ \phi_i^{-1}(z) & z \in V_i
        \end{array}.
        \right.
    \]
Let $U := \phi_F(\Delta)$ and $V := \phi_F(1/\Delta)$. Then $\phi_F \circ \phi_1^{-1}$ holomorphically conjugates $F\from U \to U$ to $f_1\from U_1 \to U_1$, and $\phi_1$ holomorphically conjugates $z^d\from 1/\Delta\to 1/\Delta$ to $f_1\from V_1\to V_1$.
Similarly, $\phi_F \circ \phi_2^{-1}$ anti-holomorphically conjugates $F\from V\to V$ to $f_2\from U_2\to U_2$, and $\phi_2$ holomorphically conjugates $z^d\from 1/\Delta\to 1/\Delta$ to $f_2\from V_2\to V_2$. 
    
For $i=1,2$,  $f_i$ continuously extend to $\partial U_i$ since $\phi_F \circ \phi_i^{-1}$ conjugates the dynamics $f_i \colon \partial U_i \to \partial U_i$ to $F \colon \Julia(F) \to \Julia(F)$. Moreover, $f_i$ is holomorphic on $U_i \cup V_i$. The complement $\phi_i(S^1)$ of $U_i \cup V_i$ is a quasicircle, and thus it has zero Lebesgue measure. Since being a quasi-regular map that is holomorphic almost everywhere, $f_i$ is holomorphic on $\mathbb{P}^1$. Since $f_i$ is conjugate to $z^d$ on $V_i$, $f_i$ is a polynomial. The Julia set $\Julia(f_i)$ is $\phi_i(\partial \Delta)$, which is a quasi-circle. Moreover, $f_i$ fixes each of the two Fatou components $\mathbb{P}^1 \setminus \Julia(f_i)$. 
Therefore $[f_i]$ is in the central component $\mathcal{H}_d$, as $\mathcal{H}_d$ is the only hyperbolic component in ${\rm poly}_d$ whose elements have quasi-circle Julia sets and the two Fatou components in the complement are fixed.
We define $\Theta^{-1}(F):=([f_1],[f_2])$.

By construction of $\Theta$ and $\Theta^{-1}$ (in particular, the construction of Beltrami coefficients), it is straightforward to see that $\Theta\circ \Theta^{-1}$ and $\Theta^{-1}\circ \Theta$ are the identity maps. Since bijective holomorphic maps are biholomorphisms, $\Theta$ is a biholomorphism.

{\it (Step 4) We show that $\Theta({\rm Diag})=\mathcal{B}_d^{fm}$}. 
By construction, for any $[f] \in \mathcal{H}_d^{fm}$, $\Theta([f],[f])$ is symmetric under $1/\bar{z}$, i.e., $(1/\bar{z})\circ \Theta([f],[f])=\Theta([f],[f]) \circ (1/\bar{z})$. Hence, $\Theta([f],[f])$ is a Blaschke product. Hence $\Theta({\rm Diag})\subset \mathcal{B}_d^{fm}$.

Fix $[F] \in \Bcal_d^{fm}$. Since $F$ is symmetric under $1/\bar{z}$, we have $\phi_1=\phi_2$ in the construction of $\Theta^{-1}([F])=([f_1],[f_2])$. Then $[f_1] = [f_2]$, and thus $[F] \in \Theta({\rm Diag})$.

Therefore, $\Theta|_{\rm Diag} \from {\rm Diag} \to \mathcal{B}_d^{fm}$ is a diffeomorphism as $\Theta$ is a biholomorphism.

(Steps 1--4) complete the proof.
\end{proof}

\subsection{Simultaneous uniformization of quasi-Blaschke products with Blaschke products (Proof of Theorem \ref{main theorem 1}).} \label{sec_holo_simul_BD}
\begin{defn}[Holomorphic structures on $\Bcal_d^{fm}$]\label{def:HoloStrB_d}
We define a holomorphic structure on $\Bcal_d^{fm}$ by pulling back the holomorphic structure of $\Hcal_d^{fm}$ via the diffeomorphism $p_1\circ \Theta|_{\rm Diag}^{-1}\from \Bcal_d^{fm} \to \Hcal_d^{fm}$, where $p_1\from \Hcal_d^{fm}\times \overline{\Hcal_d^{fm}}\to \Hcal_d^{fm}$ is the projection onto the first coordinate.
\end{defn}

Now we give a proof of Theorem \ref{main theorem 1}.
\begin{proof}[Proof of Theorem \ref{main theorem 1}]
Let $p_1\from \Hcal_d^{fm}\times \overline{\Hcal_d^{fm}}\to \Hcal_d^{fm}$ and $p_2\from \Hcal_d^{fm}\times \overline{\Hcal_d^{fm}}\to \overline{\Hcal_d^{fm}}$ be the coordinate projection maps. Then we have diffeomorphisms $p_1\circ \Theta|_{\rm Diag}^{-1}\from \Bcal_d^{fm} \to \Hcal_d^{fm}$ and $p_2\circ \Theta|_{\rm Diag}^{-1}\from \overline{\Bcal_d^{fm}} \to \overline{\Hcal_d^{fm}}$, which are respectively biholomorphic and anti-holomorphic with respect to the holomorphic structures defined in Definition \ref{def:HoloStrB_d}. 
The product of these maps gives a biholomorphism $p\from \Bcal_d^{fm}\times \overline{\Bcal_d^{fm}}\to \Hcal_d^{fm}\times \overline{\Hcal_d^{fm}}$. Then, we define a map $\mathcal{U}\from \Bcal_d^{fm}\times \overline{\Bcal_d^{fm}} \to \mathcal{QB}_d^{fm}$ by $\mathcal{U}:=\Theta\circ p$, which is a biholomorphism by Theorem \ref{thm_simul}. This proves Theorem \ref{main theorem 1}.
\end{proof}

\begin{lem}\label{lem:Equiv Btw Mating Maps}
The two maps ${\rm Mate}\from \Bcal_d^{fm}\times \Bcal_d^{fm} \to \mathcal{QB}_d^{fm}$ and $\mathcal{U}\from \Bcal_d^{fm}\times \overline{\Bcal_d^{fm}} \to \mathcal{QB}_d^{fm}$ are identical, as smooth maps. Namely, for any $[f],[g] \in \Bcal_d^{fm}$, we have ${\rm Mate}([f],[g]) = \mathcal{U}([f],[g]).$
\end{lem}
\begin{proof}
    The inverse $\mathcal{U}^{-1}$ is equal to $p^{-1} \circ \Theta^{-1}$, where $\Theta^{-1}\from \QB_d^{fm}\to \Hcal_d^{fm}\times \overline{\Hcal_d^{fm}}$ is constructed in the proof of Theorem \ref{thm_simul} and $p\from \Bcal_d^{fm}\times \overline{B_d^{fm}}\to \Hcal_d^{fm}\times \overline{\Hcal_d^{fm}}$ is constructed in the proof of Theorem \ref{main theorem 1}. Let $f$ and $g$ be the standard representatives of $[f]\in \Bcal_d^{fm}$ and $[g]\in \overline{\Bcal_d^{fm}}$, respectively. Recall that the dynamics of ${\rm Mate}([f],[g])$ on two Fatou components are holomorphically conjugate to $f\from \Delta\to \Delta$ and $g\from 1/\Delta\to 1/\Delta$. By carefully checking images of markings of the fixed points under $\mathcal{U}$ and ${\rm Mate}$, one can show that $\mathcal{U}^{-1}({\rm Mate}([f],[g]))=([f],[g])$.
\end{proof}

\section{Proofs of Theorems \ref{main theorem 2}, \ref{main theorem 3} and \ref{theorem:Path metric}} \label{sec_pf_1.2} 
We prove Theorems \ref{main theorem 2}, \ref{main theorem 3} and \ref{theorem:Path metric} in this section.
We start with preliminary lemmas in Section \ref{sec_4.1}. In Section \ref{sec_4.2}, we prove a key property about the Weil--Petersson semi-norms in relation to the complex structure of $\Bcal_d^{fm}$, namely Proposition \ref{prop_Ivrii_trick}. In Section \ref{sec_4.3}, we show that $C^1$-paths in $\Bcal^{fm}_d$ have non-zero lengths, which together with arguments in \cite[Section 5]{BH24} implies Theorem \ref{theorem:Path metric}. Finally, we prove Theorems \ref{main theorem 2} and \ref{main theorem 3} in Section \ref{sec_4.4}.

\subsection{Auxiliary lemmas} \label{sec_4.1}

\begin{defn}[(Anti-)holomorphic fixed points]\label{def_anti_hol_fixed}
Suppose that $[f]\in \Bcal_d^{fm}$ and $(f;0,\infty,1,x_4,x_5,\dots,x_{d+1})$ is its standard representative (see Definition \ref{defn:StandardRep}). By $p_1\circ \Theta|_{\rm Diag}^{-1}\from \Bcal_d^{fm}\to \Hcal_d^{fm}$, the attracting fixed point $0$ of $f$ is mapped to the attracting fixed point in $\Cbb$ of the polynomial $p_1\circ \Theta|_{\rm Diag}^{-1}([f])\in \Hcal_d^{fm}$. Therefore, the multiplier of $0$ is holomorphic with respect to the holomorphic structure of $\Bcal_d^{fm}$. By the $1/\overline{z}$-symmetry of Blaschke product, the multiplier of $\infty$ is anti-holomorphic with respect to the holomorphic structure of $\Bcal_d^{fm}$. We call $0$ the {\it holomorphic fixed point} of $f$ and $\infty$ the {\it anti-holomorphic fixed point}.

For any $[f] \in \mathcal{B}_d^{{fm}}$, the \emph{holomorphic attracting multiplier} 
$\lambda_{h.att}(f)$ is the multiplier of the holomorphic attracting fixed point of $f$.
The multiplier of the anti-holomorphic attracting fixed point is then equal to 
$\overline{\lambda_{h.att}(f)}$, the complex conjugate of $\lambda_{h.att}(f)$.
\end{defn}

\begin{lem} \label{lem_4.5}
Let $([f_t])_{t\in(-1,1)}$ be a smooth path in $\Bcal^{fm}_d$, and let $\vec{v}\in T_{[f_0]}\Bcal^{fm}_d$ be its tangent vector at $t=0$. Let $r(t)$ and $\theta(t)$ be such that
$$\lambda_{h.att}([f_t]) =r(t) e^{i\theta(t)}.$$ 
If $||\vec{v}||_{WP} = 0$, then for any integer $n \ge 1$, we have
\begin{equation} \label{eq_dot_r}
\dot{r}\left(2r^n-r^{2n}\cos{n\theta}-\cos{n\theta} \right)=-r(1-r^{2n})\dot{\theta}\sin{n\theta},
\end{equation}
where $\dot{r}:=\left.\frac{d}{dt}\right|_{t=0}r(t)$, $\dot{\theta}:=\left.\frac{d}{dt}\right|_{t=0}\theta(t)$, $r:=r(0)$, and $\theta:=\theta(0)$.
\end{lem}

\begin{proof}
By Lemma \ref{lem_degen}, for any repelling $n$-cycle $C$ with $n\ge1$, we have
$$\left.\frac{d}{d t}\right|_{t=0} \lambda_C([f_t])=0$$ where $\lambda_{C}$ is the multiplier of the cycle $C$.
Then, by taking derivative of Equation \eqref{eqn:HoloIndex} with respect to $t$, we obtain
\[
    0=\left.\frac{d}{d t}\right|_{t=0} \frac{1-r(t)^{2n}}{|1-r(t)^n e^{ni \theta(t)}|^2}=\left.\frac{d}{d t}\right|_{t=0} \frac{1-r(t)^{2n}}{1+r(t)^{2n}-2r(t)^n\cos(n\theta(t))}.
\]
By evaluating the derivative in the right-hand side, we obtain Equation \eqref{eq_dot_r}.
\end{proof}

\begin{lem} \label{lem_cases}
If Equation \eqref{eq_dot_r} holds for every $n\ge1$, then one of the following three holds. 
\begin{enumerate}
    \item $r\neq0$, $\dot{r}=0$, and $\theta\in\{0,\pi\}$ with $\dot{\theta}\neq0$.
    \item $r\neq 0$, $\dot{r}=0$, and $\dot{\theta}=0$.
    \item $r=0$ and $\dot{r}=0$.
\end{enumerate}
\end{lem}
\begin{proof}
First, we prove $\dot{r} =0$ by contradiction. Assume that $\dot{r} \neq 0$.

Consider the case when $\theta\in \mathbb{Q}\pi$. Then, there is an infinite subset $S\subset \mathbb{N}$ such that $\sin n\theta= 0$ for any $n\in S$. It follows that $\cos n\theta=\pm 1$ for any $n\in S$. By Equation \eqref{eq_dot_r}, with $\dot{r}\neq 0$, we have $2r^n-r^{2n}\cos n\theta - \cos n\theta = 2r^{n} \pm r^{2n} \pm 1= 0$ for any $n\in S$. This implies that $r^{n}=\pm 1$ for any $n \in S$, and it contradicts $r<1$.

Now, consider the case when $\theta \notin \mathbb{Q}\pi$. In this case, $\sin n\theta \neq 0$ for any non-zero integer $n$. Hence, we have
\begin{equation}\label{eqeq}
\dot{r} \left( \frac{2r^n}{\sin n\theta} - r^{2n} \cot n\theta -\cot n\theta\right) = -r\dot{\theta}+r^{2n}\dot{\theta}.
\end{equation}
Furthermore, $\left\{n \frac{\theta}{\pi}\right\}=n\frac{\theta}{\pi}- [n\frac{\theta}{\pi}]$ is equidistributed between $0$ and $1$. 
Especially, there exist a constant $c>0$ and an infinite sequence $\left\{n_{k}\right\}$ of positive integers satisfying the following: For each $k\ge 1$,
\begin{enumerate}
    \item[(i)] $\sin(n_k\theta) > c>0$  and
    \item[(ii)] $|r\dot{\theta} - \dot{r}\cot n_{k}\theta|>c>0$.
\end{enumerate}
By (i), we have $|\cot n_{k}\theta|<1/c$ for any $k\ge 1$. Therefore, since $0<r<1$ and $n_k\to \infty$ as $k\to\infty$, there exist $\epsilon \in (0,c)$ and $N$ such that for any $k>N$, the following hold:
\[
    |r^{2n_{k}}\dot{\theta}| <\epsilon/2
    ~~~{\rm and}~~~
    \left|\dot{r}\left(\frac{2 r^{n_k}}{\sin n_k\theta} - r^{2 n_k} \cot n_k\theta -\cot n_k\theta\right)+\dot{r}\cot n_{k}\theta\right| < \epsilon/2.
\]
Then, it follows from Equation\eqref{eqeq} that
\[
    \left|-r\dot{\theta}+\dot{r}\cot n_{k}\theta\right| <\epsilon<c
\]
for all $k>N$. This contradicts (ii), which states that $|\dot{r}\theta-\dot{r}\cot n_{k}\theta|>c$. Therefore, we have $\dot{r}=0$.

As $\dot{r} =0$, the left-hand side of Equation \eqref{eq_dot_r} equals 0 for all $n\ge1$. Therefore, we have two cases: $r=0$ and $r\neq 0$. If $r=0$, then we obtain Case (3) in the statement. If $r\neq0$, then either $\sin n\theta = 0$ for all $n\ge1$ or $\dot{\theta} = 0$, which correspond to Cases (1) and (2) in the statement, respectively.
\end{proof}

\begin{lem}\label{lem_perp}
Let $U$ be a open set in $\mathbb{C}^n$ and $f\from U \to \mathbb{C}$ be a function that is holomorphic or anti-holomorphic. Let $J$ and $J'$ be the almost complex structures on the real tangent bundles $TU$ and $T\mathbb{C}$, respectively. Then, for any $x\in U$ and $\vec{v}\in T_x U$, we have $D_xf(\vec{v})\cdot D_x f(J\vec{v})= 0$. Here the inner product is the Euclidean inner product on $\mathbb{C}$. 
\end{lem}
\begin{proof}
    Since $f$ is holomorphic { (resp.\@ anti-holomorphic)}, we have $D_xf(J\vec{v})=J'\cdot D_xf(\vec{v})$ (resp.\@ $D_xf(J\vec{v})=-J'\cdot D_xf(\vec{v})$). Recall that $J'$ acts as the $\pi/2$-rotation, i.e.,
    \[
        J' \cdot \partial_x =\partial_y ~~~{\rm and}~~~ J' \cdot \partial_y =-\partial_x.
    \]
    Therefore, the vector $J'\cdot D_xf(\vec{v}) \in \mathbb C$ (resp.\@ ${ -} J'\cdot D_xf(\vec{v})$) is perpendicular to the vector $D_xf(\vec{v}) \in \mathbb C$. Hence the desired conclusion follows.
\end{proof}

\subsection{Weil--Petersson semi-norm and almost complex structure} \label{sec_4.2}

\begin{lem}\label{lem_vf_J}
Let $\eta_{\Vec{v}}(z)$ be the holomorphic vector field on $\Delta$ defined by Equation \eqref{def_eta_vf}. Let $J$ be the almost complex structure on $\Bcal_d^{fm}$ given by Definition \ref{def:HoloStrB_d}. Then we have $\eta_{J\cdot \vec{v}}(z) =  -i \eta_{\vec{v}}(z)$, for all $z \in \Delta$.
\end{lem}
\begin{proof}
Let $([f_t])_{t \in \Delta}$ be a holomorphic family in $\Bcal_d^{fm}$. Then, by Theorem \ref{main theorem 1} and Lemma \ref{lem:Equiv Btw Mating Maps}, the family of matings $([F_t]:={\rm Mate}([f_0],[f_t]))_{t \in \Delta}$ is an {anti-}holomorphic family in ${\rm rat}_d^{fm}$. Assume that $f_t$'s and $F_t$'s are standard representatives.

For each $t$, there exists a holomorphic embedding $H_t \from \Delta \to \mathbb{P}^1$ such that $H_t\circ F_0 = F_t\circ H_t$ and satisfies $H_0(z) = z$ such that $\eta_{\Vec{v}}(z) := {d}/{dt}|_{t=0}H_t(z)$; see Section \ref{sec_Curt's_metric}. We note that $H_t$ extends to a quasiconformal conjugacy in a small neighborhood $N(\Delta)$ of the closure of the unit disk $\Delta$.

We claim that $H_t$ has anti-holomorphic dependence on $t$. Denote by $\phi_{0,t}\from \mathbb{P}^1 \to \mathbb{P}^1$ the quasiconformal map $\phi_{f_0,f_t}$ in Theorem \ref{thm_simul}. Recall that $\phi_{0,t}$ conjugates $z^d$ to $F_t$ near the Julia sets, and it has anti-holomorphic dependence on $t$. Hence, $\psi_t := \phi_{0,t}\circ \phi_{0,0}^{-1}$ conjugates $f_0$ to $F_t$ near the Julia sets, and it has anti-holomorphic dependence on $t$. Since $F_t$ restricted to the Fatou component containing $0$ is holomorphically conjugate to $f_0|_{\Delta}\from \Delta\to\Delta$, the $\psi_t$ on $\Delta$ coincides with $H_t$ on $\Delta$. Hence, $H_t$ is anti-holomorphic with $t$.

Let $\nu(z)$ be the Beltrami differential on $\Delta$ representing the derivative of the family of quasiconformal mappings $H_t$ at $t=0$. That is, for $\mu_t:=\partial_{\bar{z}}H_t/\partial_{z}H_t$, we have $\mu_t(z)=\overline{t}\nu(z)+\overline{t}^2\epsilon_t(z)$ with $\nu(z),\epsilon_t(z)\in L^{\infty}(\Cbb)$ for $t\in \Delta$ near $0$. Then, by \cite[Theorem 4.37]{ImaTan},  we have
\[
    \eta_{\vec{v}}(\zeta) = \frac{-1}{\pi}\iint_{\mathbb C} \nu(z)\frac{\zeta(\zeta-1)}{z(z-1)(z-\zeta)}\,dxdy, \zeta \in \mathbb C.
\]
Consider an anti-holomorphic family of quasiconformal mappings $(H'_t)_{t \in \Delta}$ defined by $H'_t:=H_{i\cdot t}$. Then, for $\mu'_t:=\partial_{\bar{z}}H'_t/\partial_{z}H'_t$, we have $\mu'_t(z)=\mu_{i\cdot t}(z)=\overline{i t}(\nu(z))+\overline{i t}^2(\epsilon_t(z))=-\overline{t}(i\nu(z))+\overline{t}^2(-i\epsilon_t(z))$. Hence, we have $\eta_{J\cdot \Vec{v}}(z) = - i \eta_{\Vec{v}}(z)$, for all $z \in \Delta$.
\end{proof}

\begin{prop}[Ivrii's trick]\label{prop_Ivrii_trick}
For any $[f] \in \Bcal_d^{fm}$ and $\vec{v} \in T_{[f]}\Bcal_d^{fm}$, we have $||\Vec{v}||_{WP} =||J\cdot \Vec{v}||_{WP}$, where $J$ is the almost complex structure on $\Bcal^{fm}_d$.     
\end{prop}

\begin{proof}
By Lemma \ref{lem_vf_J} and the definition of $||\cdot||_{WP}$ as in \eqref{eq_def_metric}, we have $||\vec{v}||_{WP} = ||J\cdot \vec{v}||_{WP}$ for any $v \in T_{[f]}\Bcal_d^{fm}$. 
\end{proof}

\subsection{Positivity of lengths of $C^1$-paths and Theorem \ref{theorem:Path metric}} \label{sec_4.3}
Given a $C^1$-path $\gamma \colon [0,1]\to \Bcal^{fm}_d$, we define the {\it Weil--Petersson length} $\ell_{WP}(\gamma)$ of $\gamma$ as
$$\ell_{WP}(\gamma) := \int_0^1 ||\gamma'(t)||_{WP} \,dt.$$

As mentioned in the introduction, the following lemma, together with the analyticity of $||\cdot||_{WP}$, implies that $||\cdot||_{WP}$ defines a path metric (Theorem \ref{theorem:Path metric}).
We also note that the following lemma is analogous to \cite[Theorem 9.1]{Smania25}, which shows that a smooth path of
orientation-preserving expanding circle maps having zero pressure length
has smoothly conjugate endpoints. In our setting, we use marked
multiplier rigidity together with rigidity of conjugacies between
Blaschke products.

\begin{lem}\label{p:non-deg-analytic-paths}
We have $\ell_{WP}(\gamma)>0$ for every non-constant $C^1$-path $\gamma \colon [0,1]\to \Bcal^{fm}_d$.
\end{lem}

\begin{proof}
Suppose that $\gamma \colon [0,1]\to \Bcal^{fm}_d$ is a non-trivial $C^1$-path such that $\ell_{WP}(\gamma)=0$. Let $[f_t]:=\gamma(t)\in \Bcal^{fm}_d$. Then by Lemma \ref{lem_degen}, all the repelling multipliers
of $[f_t]$ are constant along $\gamma$. In particular, for any $t_1,t_2 \in (0,1)$, $[f_{t_1}]$ and $[f_{t_2}]$ have the same repelling multipliers. By \cite[Theorem 10.2.1]{PU}, the quasisymmetric conjugacy $\phi\colon J(f_{t_1}) \to J(f_{t_2})$
is Lipschitz and so is its inverse. In particular, $\phi$ is {\it absolutely continuous}. 
Then, by \cite[Theorem 4]{ShubSullivan85}, $\phi$ is the restriction of a M\"obius transformation. Therefore $f_{t_1}$ and $f_{t_2}$ are M\"obius conjugate. Hence, $\gamma$ is a trivial path. This contradiction completes the proof.
\end{proof}

We refer to \cite[Section~5]{BH24} for the non-degeneracy of the length, $\ell(\gamma) > 0$, in certain deformation spaces different from $\mathcal{B}_d^{fm}$.

\subsection{Proofs of Theorem \ref{main theorem 2}} \label{sec_4.4}

\begin{proof}[Proof of Theorem \ref{main theorem 2}]
    Theorem \ref{main theorem 2}-(2) is a consequence of Lemma \ref{p:non-deg-analytic-paths}.
    
    Let us prove Theorem \ref{main theorem 2}-(1), which states that $||\vec{v}||_{WP}= 0$ implies $\vec{v}\in T\mathcal{A}_d^{fm}(\lambda)$ for some $\lambda\in \Delta$.
    Suppose $||\vec{v}||_{WP}= 0$. Recall the three cases (1), (2), and (3) in Lemma \ref{lem_cases}. Case (3) is equivalent to $\vec{v}\in T\mathcal{A}_d^{fm}(0)$, and Case (2) is equivalent to $\vec{v}\in T\mathcal{A}_d^{fm}(\lambda)$ for $\lambda\neq 0$. Therefore, it suffices to show that Case (1) cannot happen.

    Suppose that $\vec{v}$ satisfies the condition in Case (1) of Lemma \ref{lem_cases}. Recall that the multiplier function of the holomorphic fixed points $\lambda_{h.att}\from \Bcal_d^{fm} \to \mathbb{C}$ is holomorphic with respect to $J$ on $\Bcal_d^{fm}$; see Definition \ref{def_anti_hol_fixed}. Let $([f_t])_{t\in (-1,1)}$ be a smooth path in $\Bcal_d^{fm}$ with $\vec{v}=\left.\frac{d}{dt}\right|_{t=0}[f_t]$. Observe that $r \neq 0$ is equivalent to $\lambda_{h.att}([f_0]) \neq 0$ and $\dot{r} = 0$ is equivalent to $(D_{[f_0]}|\lambda_{h.att}|)(\vec{v})=0$.  Then $D_{[f_0]} \lambda_{h.att}(\vec{v})$ is perpendicular to $\lambda_{h.att}([f_0])$ as vectors in $\mathbb{R}^2$ with the Euclidean inner product. Note that $D_{[f_0]} \lambda_{h.att}(\vec{v})\neq0$ because $\dot{\theta}\neq 0$.

    It follows from Proposition \ref{prop_Ivrii_trick} that we have $||J\cdot \vec{v}||_{WP}=0$. Applying Lemma \ref{lem_cases} to $J\cdot \vec{v}$, we see that $||J\cdot \vec{v}||_{WP}=0$ implies that $(D_{f_0}|\lambda_{h.att}|)(J \cdot \vec{v})=0$ (as $\lambda_{h.att}(f_0) \neq 0$). Hence, similarly to the previous paragraph, we have that $D_{[f_0]} \lambda_{h.att}(J\cdot \vec{v})$ is non-zero and perpendicular to $\lambda_{h.att}([f_0])$ as vectors in $\mathbb{R}^2$.

    However, by Lemma \ref{lem_perp}, $D_{[f_0]}\lambda_{h.att}(J\cdot \vec{v})$ is perpendicular to $D_{[f_0]}\lambda_{h.att}(\vec{v})$. This is a contradiction, because two non-zero vectors, $D_{[f_0]}\lambda_{h.att}(\vec{v})$ and $D_{[f_0]}\lambda_{h.att}(J \cdot \vec{v})$, which are perpendicular to each other in $\mathbb{R}^2$, cannot both be perpendicular to another non-zero vector $\lambda_{h.att}([f_0])$ in $\mathbb{R}^2$. Therefore, Case (1) of Lemma \ref{lem_cases} cannot happen.
\end{proof}

\subsection{Some calculations on the non-degeneracy of $(\mathcal{B}_3^{fm},||\cdot||_{WP})$} \label{subsec:deg3}
    Each element in $\mathcal{A}^{fm}_3(\lambda)$ has a standard representative of the form
    \begin{equation}\label{eqn:deg3}
        f_{a,b}(z)=\frac{\overline{a}+1}{1+a}\frac{\overline{b}+1}{1+b} z\frac{z+a}{\overline{a}z+1}\frac{z+b}{\overline{b}z+1},
    \end{equation}
    where $a,b\in \Delta$ with
    \begin{equation}\label{eqn:Att_Multi_Cubic_Blaschke}
        \lambda=\frac{ab(1+\bar{a})(1+\bar{b})}{(1+a)(1+b)}.
    \end{equation}
    
    \noindent {\it Case 1: $\lambda=0$.}\quad Without loss of generality, we suppose $b=0$. Define
    \[
        f_a(z):=\frac{\overline{a}+1}{1+a} z^2\frac{z+a}{\overline{a}z+1},
    \]
    where $a\in \Delta$. For the multiplier $\lambda_1(f_a)$ at $1$, we have
    \[
        \lambda_1(f_a)=\frac{1}{1+a}+\frac{1}{1+\overline{a}}.
    \]

    Suppose that there exists a tangent vector $v\in T_{[f_a]}\mathcal{A}^{fm}_3(0)$ with $||v||_{WP}=0$. By Proposition \ref{prop_Ivrii_trick}, we have $||J\cdot v||_{WP}=0$. Since $\dim_{\mathbb{R}}(T_{[f_a]}\mathcal{A}^{fm}_3(0))=2$, it follows $||w||_{WP}=0$ for every $w\in T_{[f_a]}\mathcal{A}^{fm}_3(0)$. Then, by Lemma \ref{lem_degen}, we have $D_{[f_a]}\lambda_1=0$. This contradicts $\partial_a\lambda_1=-\frac{1}{1+a^2}\neq 0$ and $\partial_{\overline a}\lambda_1=-\frac{1}{1+\overline{a}^2}\neq 0$.
    
    \noindent {\it Case 2: $\lambda\neq0$.}\quad  In this case, $a,b\neq 0$. By taking $\frac{\partial}{\partial a}$ and $\frac{\partial}{\partial \bar{a}}$ to \eqref{eqn:Att_Multi_Cubic_Blaschke}, we obtain
    \[
    \left\{
        \begin{aligned}
            \frac{1+\bar{b}}{1+b}\frac{\partial b}{\partial a}+b\frac{\partial \bar{b}}{\partial a}&=-\frac{b(1+\bar{b})}{a(1+a)}\\
            \frac{1+\bar{b}}{1+b}\frac{\partial b}{\partial \bar{a}}+b\frac{\partial \bar{b}}{\partial \bar{a}}&=-b\frac{1+\bar{b}}{1+\bar{a}}.
        \end{aligned}
    \right.
    \]
    By using
    \begin{equation}\label{eqn:CPlx_Deri}
        \overline{\frac{\partial \bar{b}}{\partial a}}=\frac{\partial b}{\partial \bar{a}}~~~~ {\rm and}~~~~ \overline{\frac{\partial b}{\partial a}}=\frac{\partial \bar{b}}{\partial \bar{a}},
    \end{equation} we obtain
    \begin{equation}\label{eqn:Deri about a 1}
        \frac{\partial b}{\partial a}=-\frac{b(1+b)(1-a\bar{b})}{a(1+a)(1-|b|^2)}~~~~{\rm and}~~~~\frac{\partial \bar{b}}{\partial a} =\frac{(b-a)\bar{b}(1+\bar{b})}{a(1+a)(1-|b|^2)}.
    \end{equation}
    For the multiplier $\lambda_1(f_a)$ at $1$, we have
    \[
        \lambda_1(f_a)=-1+\frac{1}{1+a}+\frac{1}{1+b}+\frac{1}{1+\bar{a}}+\frac{1}{1+\bar{b}}.
    \]
    Hence, by using \eqref{eqn:Deri about a 1}, we obtain
    \begin{align*}
        \partial_a\lambda_1 &=-\left(\frac{1}{1+a^2}+\frac{1}{1+b^2}\frac{\partial{b}}{\partial a}+\frac{1}{1+\bar{b}^2}\frac{\partial{\bar{b}}}{\partial a}\right)\\
        &=\frac{1}{1 + a^2} - \frac{b (1 + b) (1 - a \bar{b})}{a (1 + a) (1 + b^2) (1 - |b|^2)} - \frac{(a - b) \bar{b} (1 + \bar{b})}{a (1 + a) (1 - |b|^2) (1 + \bar{b}^2)}.
    \end{align*}
    Since $|a|,|b|<1$, we obtain that $\partial_a\lambda_1=0$ if and only if $a=b$.   
    It follows that when $a = b$ we have $\partial_a \lambda_1 = 0$, and similarly, $\partial_{\bar{a}} \lambda_1 = 0$ (along $\mathcal{A}_3^{fm}(\lambda)$). By the holomorphic index formula, the multiplier of the other repelling fixed point also has the zero derivative at $a=b$ along $\mathcal{A}_3^{fm}(\lambda)$. Therefore, to show the non-degeneracy of $||\cdot||_{WP}$ by using Lemma \ref{lem_degen}, we have to look at the derivatives of the multipliers of repelling $2$-cycles, whose explicit computation is too complicated.

    To sum up, we obtain:
    \begin{prop}\label{prop:non-deg for deg3}
        For a non-zero tangent vector $\vec{v}\in T\Bcal_3^{fm}$, we have $||\vec{v}||_{WP}\neq 0$ for every $\vec{v}\notin T_{[f_{a,a}]}\mathcal{A}_3^{fm}$ with $a\in \Delta$.
    \end{prop}

\bibliography{PressureMetric}
\bibliographystyle{plain}

\end{document}